\documentclass[12pt,a4paper]{article}

\DeclareSymbolFont{AMSb}{U}{msb}{m}{n}
\DeclareSymbolFontAlphabet{\Bbb}{AMSb}

\usepackage[ansinew]{inputenc} 
\usepackage{latexsym}
\usepackage{color}

\usepackage{amstext,amsfonts,amsmath,amsthm,graphicx,amssymb,amscd,epsfig}
\usepackage{rotating}

\usepackage{subfig}
\newtheorem{theorem}{Theorem}[section]

\newtheorem{definition}[theorem]{Definition}
\newtheorem{lemma}[theorem]{Lemma}
\newtheorem{corollary}[theorem]{Corollary}
\newtheorem{proposition}[theorem]{Proposition}

\theoremstyle{definition}

\newtheorem{example}[theorem]{Example}

\newcommand{\R}{\mathbb{R}}

\newcommand{\N}{\mathbb{N}}

\begin{document}

\begin{center}
{\Large Finite switching near heteroclinic networks} \\
\mbox{} \\
\mbox{} \\
\begin{tabular}{ccc}
S.B.S.D.\ Castro$^{1,2,*}$ & and & L.\ Garrido-da-Silva$^{1}$ \\
(sdcastro@fep.up.pt)  & & (lilianagarridosilva@sapo.pt)\\
OrcID: 0000-0001-9029-6893 & & OrcID: 0000-0003-4294-3931
\end{tabular}
\end{center}

\vspace{1cm}
$^*$ Corresponding author: sdcastro@fep.up.pt;
phone: +351 225 571 100; fax: +351 225 505 050. 
\medskip

$^{1}$
Centro de Matem\'atica da Universidade do Porto (CMUP),
Rua do Campo Alegre 687,
4169-007 Porto,
Portugal. 
\medskip

$^{2}$
Faculdade de Economia da Universidade do Porto, and
Centro de Economia e Finan\c{c}as (Cef.UP),
Rua Dr.\ Roberto Frias,
4200-464 Porto,
Portugal. 
\vspace{1cm}

\begin{abstract}
We address the level of complexity that can be observed in the dynamics near a robust heteroclinic network. We show that infinite switching, which is a path towards chaos, does not exist near a heteroclinic network such that the eigenvalues of the Jacobian matrix at each node are all real. Furthermore, for a path starting at a node that belongs to more than one heteroclinic cycle, we find a bound for the number of such nodes that can exist in any such path.
This constricted dynamics is in stark contrast with examples in the literature of heteroclinic networks such that the eigenvalues of the Jacobian matrix at one node are complex.
\end{abstract}



\vspace{1.5cm}

\section{Introduction}

We are concerned with describing the level of complexity that can be achieved by the dynamics near a robust heteroclinic network. 
We work in the context of smooth vector fields in finite-dimensional Euclidean space. A heteroclinic network is a connected union of finitely many heteroclinic cycles, consisting of equilibria connected by heteroclinic trajectories. See Section~\ref{sec:preliminary} for precise definitions.
Authors such as\footnote{Our references in this section do not pretend to be comprehensive. They are rather meant as a starting point for the curious reader.} Aguiar {\em et al.} \cite{ACL2005, ACL2006}, Labouriau and Rodrigues \cite{LabRod2017, RodLab2014} prove the existence of chaotic dynamics near heteroclinic networks such that the eigenvalues of the Jacobian matrix at, at least, one node are complex. Chaotic dynamics are induced by the existence of infinite switching (see \cite{ACL2005} and Section~\ref{sec:preliminary}), that is, for any given sequence of heteroclinic connections in the heteroclinic network there exist initial conditions near the heteroclinic network whose trajectory remains close to that sequence. A natural question follows: that of knowing whether infinite switching can arise in heteroclinic networks such that the eigenvalues of the Jacobian matrix at every node are real. A negative answer, in particular cases, may be found in Garrido-da-Silva and Castro \cite{GSC_RSP} for the 2-person Rock-Scissors-Paper game; in Aguiar \cite{Aguiar2011} and Castro and Lohse \cite{CastroLohse2016} absence of infinite switching is proven for heteroclinic networks with a common connection between heteroclinic cycles. These latter results do not apply to heteroclinic networks with no common connection between heteroclinic cycles such as the bowtie network, as pointed out by \cite{CastroLohse2016}, although these authors do provide a hint at the absence of infinite switching in this network.
Here we address the general case. 

We prove that infinite switching does not occur near a heteroclinic network such that the eigenvalues of the Jacobian matrix at each node are all real, 
with the further common assumption that the dynamics along a heteroclinic connection, described by a global map, are a rescaled permutation.
The proofs use the geometry of the domains and their images under the local and global maps that make up the Poincar\'e return map to a cross section to the flow. 
We focus on the wide class of quasi-simple heteroclinic networks, defined in Section~\ref{sec:preliminary}, but our results apply to any heteroclinic network provided the geometric constrains we identify are satisfied. For the reader concerned with applications we stress that replicator dynamics in game theory and population dynamics naturally exhibits quasi-simple heteroclinic 
sub-networks\footnote{When one of the heteroclinic connections in a heteroclinic network contains more than one heteroclinic trajectory, we talk about a sub-network for each heteroclinic trajectory.}
satisfying the additional assumption on the global maps, 
as pointed out in Castro {\em et al.} \cite{CFGdSL}. Our results apply also to sub-networks of the ac-heteroclinic networks studied in Podvigina {\em et al.} \cite{PCL2020}, to the heteroclinic networks considered by Afraimovich {\em et al.} \cite{AfrMosYou}, as well as to all simple heteroclinic networks of types B, C, and Z that first appear in the work of Krupa and Melbourne \cite{KrupaMelbourne2004} for the first two types, and Podvigina \cite{Podvigina2012}, for the latter. See also Bick and Lohse \cite{BicLoh2019} for an example in the context of coupled cell networks.
Simple heteroclinic networks of type A do not satisfy our assumptions.
%

%
A heteroclinic network always has at least one node belonging to more than one of its heteroclinic cycles. Such a node we call a distribution node. We find a bound for the number of distribution nodes that can belong to any heteroclinic path that is followed by a trajectory near the heteroclinic network.
%
This also informs about how infinite switching fails to exist.

The absence of infinite switching is by no means synonymous with uninteresting dynamics as is patent, for example, in the work of Bick \cite{Bic2018}, Postlethwaite and Dawes \cite{PosDaw2010}, Rabinovich {\em et al.} \cite{Rab_etal2010}, and Sato {\em et al.} \cite{SAC}. Our results, while assertively answering the question of infinite switching near a heteroclinic network, open the way to the study of other interesting and possibly complicated dynamics. Our findings also inform the scientist in the quest of chaos in a dynamical system possessing a heteroclinic network such that no eigenvalue of the Jacobian matrix is complex that, if chaos is to be found, it will be at a distance from such a heteroclinic network.

This article proceeds as follows: the next section 
contains previous 
results and definitions. Section~\ref{sec:dynamics} is divided into three subsections addressing the geometry of the local and global maps between cross sections to the flow, the proof of absence of switching and constraints on switching, respectively, ending with some illustrative examples. 
In Section~\ref{sec:final} we make some relevant final remarks.

\section{Background}\label{sec:preliminary}

Let $f:\R^n \rightarrow \R^n$ be smooth and consider the dynamics described by the following ODE in $n \in \N$ state variables
\begin{equation}\label{eq:ODE}
\dot{x}=f(x).
\end{equation}

A \emph{heteroclinic cycle} \textbf{\textsf{C}} is a topological circle of hyperbolic saddle equilibria $\xi_j$, $j=1,\ldots,m$ ($m \in \N$) of 
system \eqref{eq:ODE} and heteroclinic connections $\kappa_{j,j+1}=\left[\xi_j \rightarrow \xi_{j+1}\right] \subseteq W^\text{u}(\xi_j) \cap W^\text{s}(\xi_{j+1}) \neq \emptyset$, where $W^\text{u}(\cdot)$ and $W^\text{s}(\cdot)$ denote the unstable and stable manifolds; and  $\xi_{m+1} =\xi_1$. Hereafter we also refer to~$\xi_j$ as \emph{nodes}.

It is well-known, see \cite{KrupaMelbourne2004}, that the heteroclinic cycle \textbf{\textsf{C}} is \emph{robust} if, for each $j=1,\cdots,m$, there exists a flow-invariant space $P_j \subset \R^n$ such that 
\begin{enumerate}
\item[(i)] $\kappa_{j,j+1}\subset P_j$, and
\item[(ii)] $\xi_j$ is a saddle and $\xi_{j+1}$ is a sink in~$P_{j}$.
\end{enumerate}

Let $\hat{L}_j \subset \R^n$ be the smallest flow-invariant space connecting $\xi_j$ to the origin. 
Let $X \ominus Y$ stand for the orthogonal complement of~$Y$ in~$X$.
Clearly, $\xi_j \in P_{j-1} \cap P_{j}$.
We divide the eigenvalues of the Jacobian $\textnormal{d}f(\xi_j)$ into four classes: \emph{radial} (eigenvectors belonging to $\hat{L}_j$), \emph{contracting} (eigenvectors belonging to 
$P_{j-1} \ominus \hat{L}_j$), \emph{expanding} (eigenvectors belonging to 
$P_j \ominus \hat{L}_j$) and \emph{transverse} (all others). 

The above classification of the eigenvalues is \emph{local} in nature as it refers to the role of each eigenvalue at a node.
We use the formulation local-contracting, etc, when we need to clarify that we
are working with 
the local classification.
When looking at the whole heteroclinic network we use the convention in \cite{PCL2020}: by global-contracting (resp.\ global-expanding) eigenvalues at $\xi_j$ in \textbf{\textsf{N}} we mean the contracting (resp.\ expanding) eigenvalues of all heteroclinic cycles that $\xi_j$ belongs to; by global-transverse eigenvalues at $\xi_j$ we mean the eigenvalues that are not radial, 
global-contracting or 
global-expanding for any heteroclinic cycle through $\xi_j$. In what follows 
the distinction between local-transverse eigenvalues and global-transverse eigenvalues is important. 
We refer mostly to the global classification of the eigenvalues and omit the word {\em global} when there is no risk of confusion.
We note that the sign of the global-transverse eigenvalues is not determined {\em a priori}.

For the reader familiar with the heteroclinic network studied by Kirk and Silber \cite{KS}, see Figure~\ref{fig:KS_ac}~(a) below, at $\xi_2$ the eigenvalue $e_{24}$ (resp.\ $e_{23}$) is local-transverse for the heteroclinic cycle containing $\xi_3$ (resp.\ $\xi_4$) but is global-expanding when considering the whole heteroclinic network. At $\xi_2$ there are no global-transverse eigenvalues. Either at $\xi_3$ or at $\xi_4$ the (unique) local-transverse eigenvalue is global-transverse.

\begin{definition}[Definition 2.1 in \cite{GdSC}]
A \emph{quasi-simple} cycle is a robust heteroclinic cycle \textbf{\textsf{C}} connecting $m\in\N$ equilibria 
$\xi_j \in P_{j-1} \cap P_{j}$
so that for all $j=1, \ldots, m$
\begin{itemize}
\item[(i)] $P_j$ is a flow-invariant space,
\item[(ii)] $\textnormal{dim}(P_j) =  \textnormal{dim}(P_{j+1})$,
\item[(iii)] $\textnormal{dim}(P_j \ominus \hat{L}_j) = 1$
\end{itemize}
\end{definition}

Each node in a quasi-simple heteroclinic cycle has a unique local-expanding and a unique local-contracting eigenvalue,  so both local-contracting and local-expanding eigenvalues are real. 
Heteroclinic connections are all one-dimensional and the number of local-transverse and radial eigenvalues are the same for all nodes. 

A (robust) \emph{heteroclinic network} \textbf{\textsf{N}} is a finite connected union of (robust) heteroclinic cycles.
We call \emph{quasi-simple heteroclinic network} to a heteroclinic network consisting of quasi-simple heteroclinic cycles.

Some heteroclinic cycles in \textbf{\textsf{N}} share at least one node. We say that a node 
having at least two outgoing connections
in \textbf{\textsf{N}} is a {\em distribution node}.

The behaviour of trajectories along the cycles near the heteroclinic network can produce increasingly complex dynamics that range from intermittency to chaos. An interesting phenomenon is characterised by the way how nearby trajectories visit parts of the heteroclinic network, which is known as \emph{switching}. 
A heteroclinic network can exhibit different forms of switching as described below.

Let \textbf{\textsf{N}} be a heteroclinic network with a finite set of nodes for the flow $\Phi_t$ generated by $f$.
Relabelling the indices if necessary, a (finite) \emph{heteroclinic path} on~\textbf{\textsf{N}} is a sequence $\left(\kappa_{j,j+1}\right)_{j\in\{1,\cdots,k\}}$ of $k \in \N$ heteroclinic connections in~\textbf{\textsf{N}} for which there is a sequence of nodes $\left(\xi_j\right)_{j\in\{1,\cdots,k+1\}}$ such that $\kappa_{j,j+1}=[\xi_j \rightarrow \xi_{j+1}]$.
An infinite heteroclinic path occurs when 
$k=\infty$.
Given an initial condition $x\in\R^n$, the trajectory $\Phi_t(x)$ \emph{follows} 
a heteroclinic path on \textbf{\textsf{N}} if,
for every neighbourhood $U$ of \textbf{\textsf{N}},
 there are sequences $(t_j)_{j\in\{0,\cdots,k\}}$, $(s_j)_{j\in\{1,\cdots,k\}}$ with $t_{j-1}<s_j<t_j$ such that $\Phi_{s_j}(x)$ belongs to a neighbourhood of a point $p_j \in \kappa_{j,j+1}$ contained in $U$ and $\Phi_{t_j}(x)$ belongs to a neighbourhood of $\xi_{j+1}$ contained in $U$.
See Figure~\ref{fig:path}.

\begin{figure}[!htb]
\centering
\includegraphics{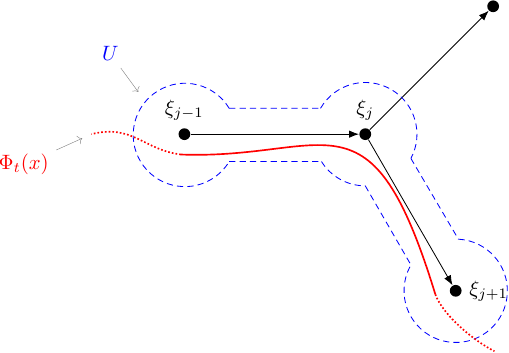}
\caption{ 
A trajectory following a heteroclinic path: given a neighbourhood $U$ of the heteroclinic path $\left[\xi_{j-1}\rightarrow \xi_j \rightarrow \xi_{j+1}\right]$, the trajectory $\Phi_t(x)$ follows it in $U$.}
\label{fig:path}
\end{figure}

\begin{definition}[\cite{ACL2005}]\label{def:swit-node}
We say there is \emph{switching at a node} $\xi_j$ of~\textbf{\textsf{N}} if, for any neighbourhood of a point in any connection leading to $\xi_j$ and sufficiently close to it, trajectories starting in that neighbourhood follow along all the possible connections forward from $\xi_j$.  
\end{definition}

\begin{definition}[\cite{ACL2005}]\label{def:swit-con}
We say there is \emph{switching along a heteroclinic connection} $\kappa_{i,j}$ of~\textbf{\textsf{N}} if, for any neighbourhood of a point in any connection leading to $\xi_i$ and sufficiently close to it, trajectories starting in that neighbourhood follow along $\kappa_{i,j}$ and then along all the possible connections forward from $\xi_j$.  
\end{definition}

\begin{definition}[\cite{ACL2005}]
We say there is \emph{finite} (resp. \emph{infinite}) \emph{switching} near~\textbf{\textsf{N}} if, for each finite (resp. infinite) heteroclinic path on \textbf{\textsf{N}}, there exists a trajectory that follows it.
\end{definition}

Of course, the occurence of infinite switching implies that of finite switching.
In addition, finite switching implies
switching along the heteroclinic connections, which in turn implies switching at all nodes of the heteroclinic network.  

The dynamics near \textbf{\textsf{N}} is investigated by composing local and global maps defined on suitable cross sections close to the equilibria. We use the notation from~\cite{CastroLohse2016}. For each~$\xi_j$ in~\textbf{\textsf{N}} let $H_j^{\text{in},i}$ be the section transverse to the incoming connection $\kappa_{i,j}$, and $H_j^{\text{out},k}$ be the section transverse to the outgoing connection $\kappa_{j,k}$. 
The \emph{local map} $\phi_{ijk}:H_j^{\text{in},i} \rightarrow H_j^{\text{out},k}$ approximates the flow near $\xi_j$ for points coming from $\xi_i$ and proceeding to $\xi_k$. The \emph{global map} $\psi_{jk}:H_j^{\text{out},k} \rightarrow H_k^{\text{in},j}$ approximates the flow along the heteroclinic connection~$\kappa_{j,k}$.

\section{Dynamics near a heteroclinic network}\label{sec:dynamics}

In this section, we show that heteroclinic networks, whose Jacobian matrices at the nodes have only real eigenvalues, do not exhibit infinite switching nearby. This is proved 
for the wide class of quasi-simple heteroclinic networks. In the first subsection we explore the geometry of the intersection of cusps in $\R^n$ and use this in the following subsection to prove our main result on the absence of infinite switching. The third, and last, subsection provides some 
constraints for the complexity of finite switching that may occur near such heteroclinic networks, 
by determining an upper bound for the number of distribution nodes in any heteroclinic path that is followed by a trajectory in the vicinity of the heteroclinic network.

\subsection{Geometry constraints}\label{sec:geometry}

We show how the existence of only real eigenvalues produces geometric constraints for the way trajectories of \eqref{eq:ODE} can evolve along a sequence of heteroclinic connections. To do this, we observe that the subsets of $H^{\text{in},i}_j$ and $H^{\text{out},k}_j$ 
that determine the direction trajectories take along heteroclinic connections are \emph{cusps}. Cusps may be \emph{thin} or \emph{thick} 
and the reader can
see their use in the context of switching in Castro and Lohse \cite[Definition~3.2]{CastroLohse2016}.  A broad definition of thin and thick cusps in $\R^2$ was provided by Podvigina and Lohse \cite[Definition~6]{PodviginaLohse2019}. We consider the latter definition when the line of reflection is a coordinate axis and extend it to $\R^n$ in the current setting. 

Here and subsequently, assume that $n\geq2$. For the reader's comfort, we reproduce Definition~6 in~\cite{PodviginaLohse2019} next.
Denote by $\ell(\cdot)$ the Lebesgue measure of a set in $\R^n$. 
For 
$\delta>0$ write 
$B_{\delta}$ for an 
$\delta$-neighbourhood of $\boldsymbol{0}\in\R^n$.  Given $\alpha>1$ and real numbers $a_1,a_2$
consider the following subset of~$\R^2$:
\[
V(a_1,a_2,\alpha) = \left\{(x_1,x_2) \in \R^2: \;\; |a_1x_1+a_2x_2| < \left( \max{\{|x_1|,|x_2|\}} \right)^\alpha \right\}.
\]

\begin{definition}[Definition 6 in \cite{PodviginaLohse2019}]
We say that $U\subset \R^2$ is a {\em thin cusp}, if
\begin{itemize}
	\item  $\ell(U \cap B_\delta) > 0$ for all $\delta > 0$;
	\item  there exist $\alpha > 1$ and $\delta > 0$ such that $U\cap B_\delta \subset V(a_1,a_2,\alpha)$, where at least one of $a_1$ and $a_2$ is not zero.
\end{itemize}
\end{definition}

The set $V(a_1,a_2,\alpha)$ has a symmetry axis that is the line $x_2=-a_1 x_1/a_2$. When this is a coordinate axis we have one of the following:  $V(0,a_2,\alpha)$ if the symmetry axis is the horizontal axis and $V(a_1,0,\alpha)$ if the symmetry axis is the vertical axis. In either case, we can simplify the set $V(a_1,a_2,\alpha)$ since, near the symmetry axis and close to the orig\text{in}, we know which of the coordinates has greatest absolute value. We thus write
\[
V(0,a_2,\alpha) = \left\{(x_1,x_2) \in \R^2: \;\; |a_2x_2| < |x_1|^\alpha \right\},
\]
and 
\[
V(a_1,0,\alpha) = \left\{(x_1,x_2) \in \R^2: \;\; |a_1x_1| < |x_2|^\alpha \right\}.
\]
See an illustration in Figure~\ref{fig:cuspsR2} which already presents the notation of Definition~\ref{def:thin-cusp}.

\begin{figure}[!htb]
\centering{}
\subfloat[]{\includegraphics[width=0.5\textwidth]{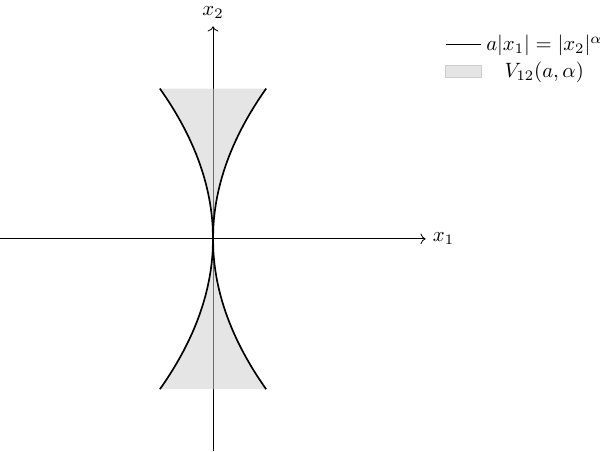}} 
\subfloat[]{\includegraphics[width=0.5\textwidth]{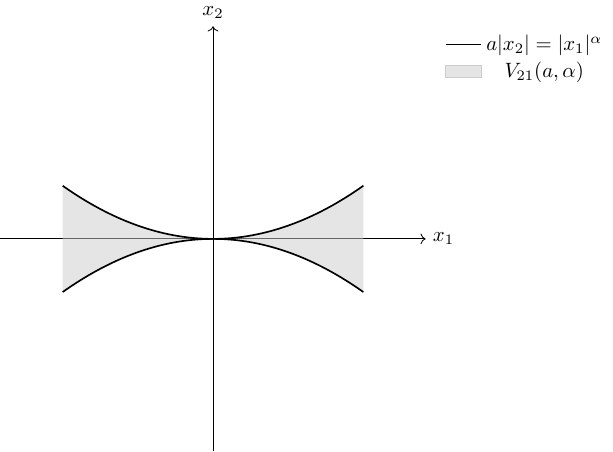}}
\caption{The two thin cusps (shaded region) in $\R^2$: (a)~$V_{21}(a, \alpha)$ and (b)~$V_{12}(a, \alpha)$. The unshaded region represents the corresponding thick cusps.}
\label{fig:cuspsR2}
\end{figure}

We generalise the definition of a thin cusp to $\R^n$, $n\geq 2$, by using the sets $V(0,a,\alpha)$ and $V(a,0,\alpha)$. This does not account for thin cusps whose symmetry axis is not one of the
coordinate axes
but suffices for our purposes.

Let $\boldsymbol{x}\equiv(x_1,\ldots,x_n)$ denote the Euclidean coordinates in $\R^n$ and
\[
\R_+^n = \left\{ \boldsymbol{x} \in \R^n: \; x_i \geq 0 \textnormal{ for all } i\right\}.
\]

\begin{definition}\label{def:thin-cusp}
A {\em thin cusp} is a subset of $\R^n$, close to the orig\text{in}, defined by 
$V_{ij}(a,\alpha) \cap B_\delta$ where
\[
V_{ij}(a,\alpha) = \left\{\boldsymbol{x} \in \R^n: \;\; 
a
|x_i| < |x_j|^\alpha \right\} 
\]
for some $0< \delta \ll 1$, $\alpha > 1$, 
$a > 0$ and $i \neq j$.
\end{definition}

The boundary of a thin cusp is a hypersurface
\[
S_{ij} (a,\alpha)= \left\{\boldsymbol{x} \in \R^n: \;\; a |x_i| = |x_j|^\alpha \right\}, 
\] 
for the same $i,j$ as in Definition~\ref{def:thin-cusp},
that is tangent to the coordinate hyperplane 
\[
\pi_i = \left\{ \boldsymbol{x}\in \R^n: \; x_i=0\right\}
\]
and symmetric about the axis 
\[
L_j =  \left\{ \boldsymbol{x} \in \R^n: \; x_k=0 \textnormal{ for all } k \neq j \right\}.
\]
In $\R^2$, the hyperplane $\pi_i$ and the axis $L_j$ naturally coincide for $i \neq j$.

The notation chosen for the thin cusp, $V_{ij}(a,\alpha)$, provides information both about which hyperplane it is tangent to and which is its symmetry line. Examples of the sets $V_{ij}(a,\alpha)$ are shown in Figure~\ref{fig:thin-cusps}. 

We define a {\em thick cusp} by reversing the inequality in the definition of a thin cusp in a slightly less generic way 
than~\cite{PodviginaLohse2019}, and write 
it $V^\text{c}_{ij}(a,\alpha)$ for short. With this definition thin and thick cusps come in pairs divided by the hypersurface $S_{ij}(a,\alpha)$ which is the boundary of the 
cusps $V_{ij}(a,\alpha)$ and $V^\text{c}_{ij}(a,\alpha)$.
In particular, 
thin and thick cusps in $\R^n$
can be defined by the connected components of $\R^n\backslash S_{ij}(a,\alpha)$.

\begin{figure}[!htb]
\centering{}
\subfloat[]{\includegraphics[scale=0.8]{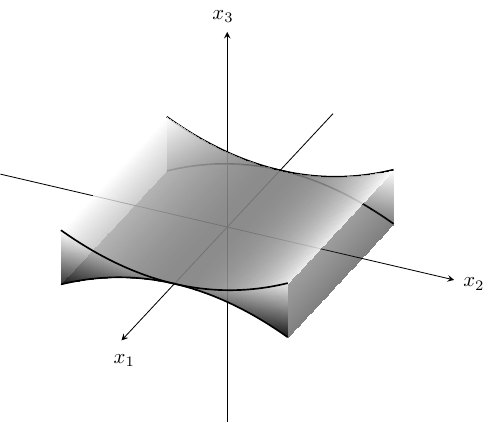}} \hfill
\subfloat[]{\includegraphics[scale=0.8]{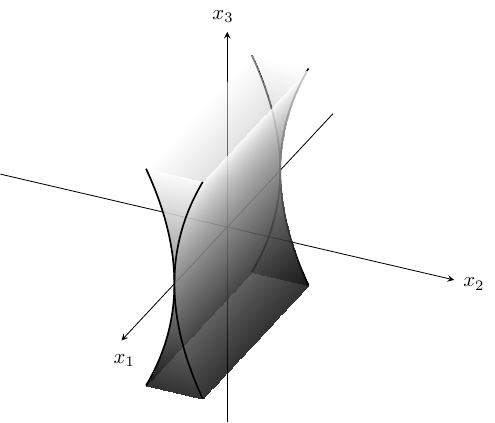}}
\caption{Examples of thin cusps (shaded region) in $\R^3$: (a)~$V_{32}(a, \alpha)$ and (b)~$V_{23}(a,\alpha)$.}
\label{fig:thin-cusps}
\end{figure}

The following results describe the intersection of groups of cusps focussing on when the intersection is empty. We say there is an {\em all-to-all intersection} of a group of cusps when each cusp intersects all the remaining th\text{in}, and corresponding thick, cusps. Obviously, a thin/thick cusp does not intersect its corresponding thick/thin cusp.

From now on we restrict the proofs to the positive orthant $\R_+^n$. 
By reflection the remaining orthants 
can be handled in much the same way.

\begin{lemma}\label{lem:2cusps}
Let there be two hypersurfaces $S_{ij}(a_i,\alpha)$ and $S_{kl}(a_k,\beta)$ defining two pairs of thin/thick cusps in $\R^n$. Then, for $B_{\delta}$~sufficiently small, two of the cusps defined by these hypersurfaces do not intersect if and only if either $i=k$ and $j=l$, or $i=l$ and $j=k$. 
\end{lemma}

\begin{proof}
The intersection of $S_{ij}(a_i,\alpha)$ with $S_{kl}(a_k,\beta)$ in~$\R_+^n$ is described by
\begin{equation}
\left\{ 
\begin{aligned}
a_i x_i & = x_j^{\alpha} \\
a_k x_k & =x_l^{\beta}.
\end{aligned}
\right.
\label{eq:intS}
\end{equation}
The system of equations \eqref{eq:intS} has at most four variables. 
If it has more than two variables then at least one of the variables $x_i$, $x_j$, $x_k$ or $x_l$ can take any real value and all cusps intersect. 

Given the definition of a cusp, two is the minimum number of variables of  \eqref{eq:intS} and it has exactly two variables when either $i=k$ and $j=l$, or $i=l$ and $j=k$. We show that, in either of these instances, at least two cusps do not intersect.

Consider that $(x_i,x_j)=(x_k,x_l)$. It means that the thin cusps $V_{ij}(a_i,\alpha)$ and $V_{kl}(a_k,\beta)$ have the same tangency hyperplane and 
symmetry axis. If $\alpha>\beta$, then 
$V_{kl}(a_k,\beta)$ contains $V_{ij}(a_i,\alpha)$ in every $B_\delta$
for $\delta>0$ sufficiently small. Therefore, the thin cusp $V_{ij}(a_i,\alpha)$ does not intersect the thick cusp 
$V^\text{c}_{kl}(a_k,\beta)$.

When $(x_i,x_j)=(x_l,x_k)$ the thin cusps $V_{ij}(a_i,\alpha)$ and $V_{kl}(a_k,\beta)$ are geometrically opposed within the $(x_i,x_j)$-subspace. While the former is arbitrarily close to the $x_i$-axis the latter is arbitrarily close to the $x_j$-axis in $B_{\delta}$. Hence, the thin cusps do not intersect near the origin.
\end{proof}

Geometrically, the solutions of \eqref{eq:intS} in $B_\delta$ describe the $(n-2)$-dimensional coordinate subspace $\left\{ \boldsymbol{x} \in \R^n: \; x_i=x_j=0 \right\}$ 
whenever the coordinates 
coincide, i.e., $(x_i,x_j)=(x_k,x_l)$ or $(x_i,x_j)=(x_l,x_k)$, for small $\delta>0$. 
The condition $i=k$ and $j=l$ corresponds to cusps that are tangent to the same coordinate hyperplane 
and symmetric about the same axis. They are nested and therefore at least one thin cusp does not intersect one thick cusp as depicted in Figure~\ref{fig:2-cusps-intersect-1}. 
The condition $i=l$ and $j=k$ is illustrated in Figure~\ref{fig:2-cusps-intersect-2}.
Otherwise, the cusps defined by $S_{ij}(a_i,\alpha)$ and $S_{kl}(a_l,\beta)$ naturally intersect all-to-all near the origin. We obtain an $(n-2)$-dimensional nontrivial hypersurface containing the orig\text{in}, see Figure~\ref{fig:2-cusps-intersect-3}. Panel (b) of Figures~\ref{fig:2-cusps-intersect-1}-\ref{fig:2-cusps-intersect-3} restricts to a neighbourhood in the positive orthant.

\begin{figure}[!htb]
\centering{}
\subfloat[$\R^3$]{\includegraphics[scale=0.85]{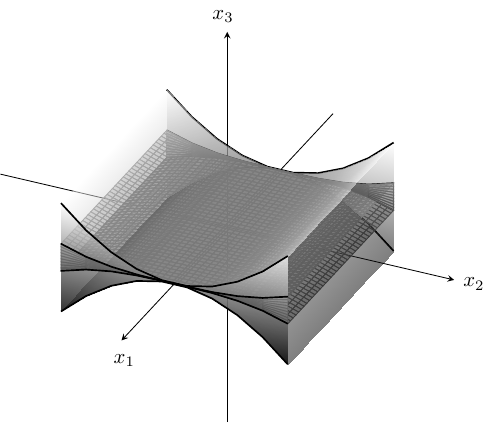}} \hfill
\subfloat[$B_\delta \cap \R^3_+$]{\includegraphics[scale=0.6]{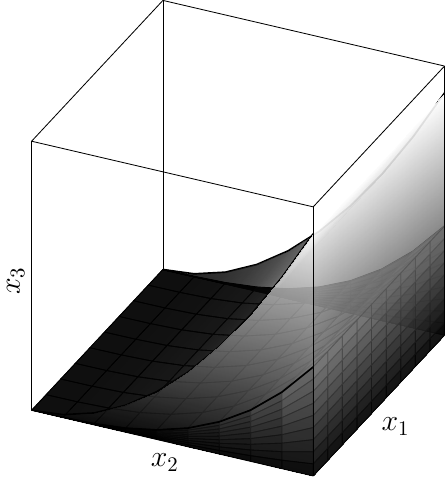}}
\caption{In $\R^3$ the two thin cusps $V_{32}(a, \alpha)$ (shaded) and $V_{32}(b, \beta)$ (grid) are nested and the corresponding thick cusp $V^{\text{c}}_{32}(a, \alpha)$ (unshaded) does not intersect $V_{32}(b, \beta)$.
}
\label{fig:2-cusps-intersect-1}
\end{figure}

\begin{figure}[!htb]
\centering{}
\subfloat[$\R^3$]{\includegraphics[scale=0.85]{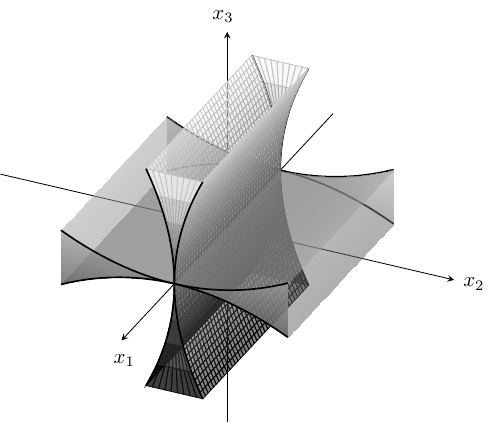}} \hfill
\subfloat[$B_\delta \cap \R^3_+$]{\includegraphics[scale=0.6]{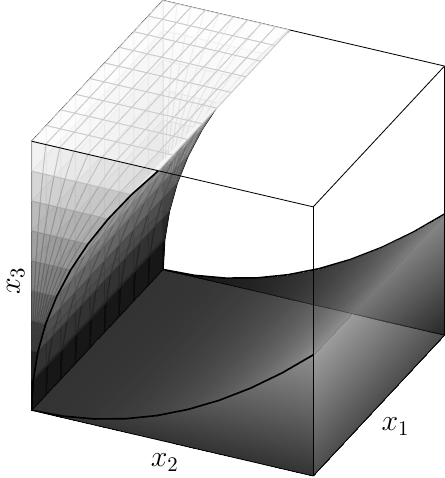}}
\caption{In $\R^3$~the two thin cusps $V_{23}(a, \alpha)$ (shaded) and $V_{32}(a, \alpha)$ (grid) do not intersect near the origin.}
\label{fig:2-cusps-intersect-2}
\end{figure}

\begin{figure}[!htb]
\centering{}
\subfloat[$\R^3$]{\includegraphics[scale=0.85]{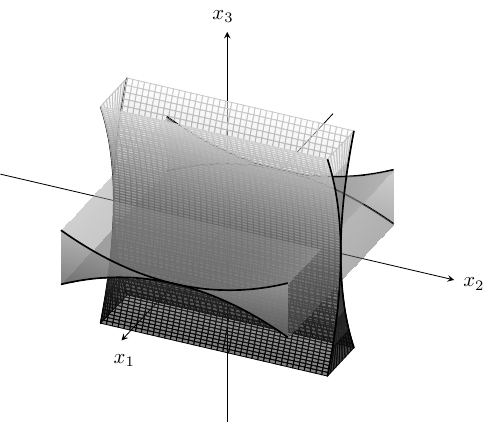}} \hfill
\subfloat[$B_\delta \cap \R^3_+$]{\includegraphics[scale=0.6]{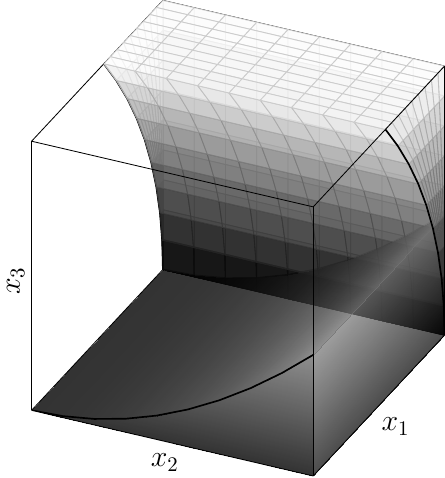}}
\caption{In $\R^3$ the two thin cusps $V_{32}(a, \alpha)$ (shaded) and $V_{13}(a, \alpha)$ (grid) and the corresponding thick cusps intersect near the origin.}
\label{fig:2-cusps-intersect-3}
\end{figure}

\begin{corollary}\label{lem:cusps}
Consider pairs of thin/thick cusps in $\R^n$ defined by $1+n(n-1)/2$ hypersurfaces $S_{ij}(a_i,\alpha_j)$ tangent to 
as many hyperplanes $\pi_i$. Then, for $B_{\delta}$~sufficiently small, at least two of the cusps do not intersect.
\end{corollary}

\begin{proof}
The proof follows from that of Lemma~\ref{lem:2cusps} since the number of 
pairs of thin/thick cusps exceeds by one the total
number of 
distinct coordinate planes in $\R^n$ given by the possible combinations of $n$ dimensions taken two at a time, i.e.
$C^n_2=n(n-1)/2$. Hence, at least two of the hypersurfaces $S_{ij}(a_i,\alpha_j)$ are defined by the same two coordinates.
\end{proof}

\begin{proposition} \label{lem:n-cusps}
In $\R^n$, $n$ pairs of cusps never intersect all-to-all for $B_{\delta}$~sufficiently small.
\end{proposition}

\begin{proof}

We proceed by induction on $n$. The result is immediately true for $n=2$.
We write the proof for three hypersurfaces defining three pairs of cusps in $\R^3$ as it provides the essential arguments. Given that there are three coordinates, $S_{ij}(a_i,\alpha_j)$ must be such that $i, j \in \{ 1,2,3 \}$. 
If the two indices of any two 
hypersurfaces are repeated, then Lemma~\ref{lem:2cusps} applies, finishing the proof. 

For the general case, there are either no repeated first indices or no repeated second indices. 
It means that every two pairwise cusps have different tangency hyperplanes and symmetry axes.
The respective three hypersurfaces $S_{j i}(a,\alpha)$, $S_{k j}(b,\beta)$ and $S_{i k}(c,\gamma)$ in $\R^3_+$ are represented by the equations:
\[
\left\{ 
\begin{aligned}
a x_j & = x_i^{\alpha} \\
b x_k & = x_j^{\beta} \\
c x_i & = x_k^{\gamma}.
\end{aligned}
\right. 
\]
The system can be reduced to two variables such that
\[
\left\{ 
\begin{aligned}
a x_j & = x_i^\alpha \\
c \, b^{\gamma} x_i & =  x_j^{\beta \gamma}.
\end{aligned}
\right. 
\]
Since $1 / \alpha < 1$ and $\beta \gamma> 1$, we know that, within the $(x_i,x_j)$-subspace, the two thin cusps 
$a x_j < x_i^\alpha$ and $c \, b^{\gamma} x_i <  x_j^{\beta \gamma}$ 
do not intersect near the origin from the base step. Substituting independently $b x_k =  x_j^\beta$ and $c x_i = x_k^{\gamma}$ into $c \, b^{\gamma} x_i <  x_j^{\beta \gamma}$, we get $c x_i <  x_k^{\gamma}$ and $b x_k <  x_j^{\beta}$, respectively. Hence, the simultaneous intersection of three thin cusps $V_{j i}(a,\alpha)$, $V_{k j}(b,\beta)$ and $V_{i k}(c,\gamma)$ in the original space $\R_+^3$ is empty for sufficiently close to the origin.
See Figure~\ref{fig:n-pair-cusps-1}.

\begin{figure}[!htb]
\centering{}
\includegraphics[scale=0.65]{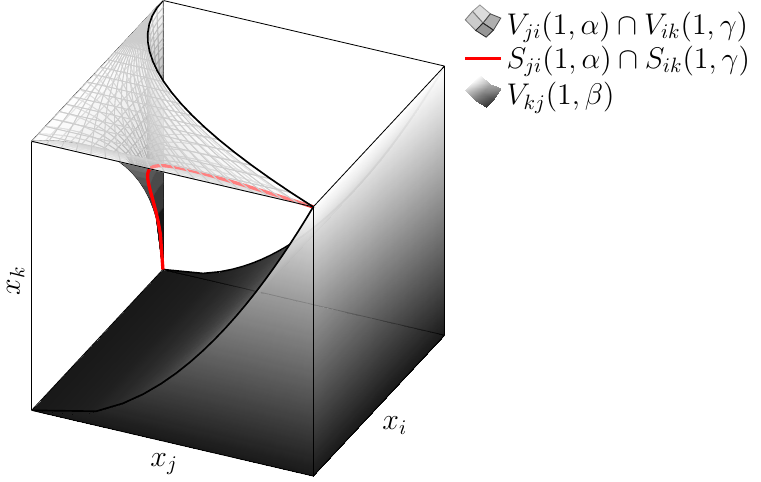}
\caption{In $\R^3$ the three thin cusps $V_{j i}(a,\alpha)$, $V_{k j}(b,\beta)$ and $V_{i k}(c,\gamma)$ do not intersect near the origin $(a=b=c=1)$.}
\label{fig:n-pair-cusps-1}
\end{figure}

Suppose now that the first indices coincide for two hypersurfaces, so they are tangent to the same hyperplane, say $x_j=0$. 
If pairs of indices do not repeat, then two hypersurfaces share the same symmetry axis, say $x_k$. The system of equations for the respective hypersurfaces $S_{j i}(a,\alpha)$, $S_{j k}(b,\beta)$ and $S_{i k}(c,\gamma)$ in $\R^3_+$ is:
\[
\left\{ 
\begin{aligned}
a x_j & = x_i^{\alpha} \\
b x_j & = x_k^{\beta} \\
c x_i & = x_k^{\gamma},
\end{aligned}
\right. 
\]
which can be reduced to 
\[
\left\{ 
\begin{aligned}
a x_j & = x_i^{\alpha} \\
c^{-1} b  \, x_j & = x_i^{\beta / \gamma}.
\end{aligned}
\right. 
\]
We need consider two cases:
\begin{enumerate}
\item[(i)] $\beta < \alpha \gamma$. 
In the $(x_i,x_j)$-subspace, the base step establishes that the two cusps $a x_j < x_i^{\alpha}$ and $c^{-1} b  \, x_j > x_i^{\beta / \gamma} $
do not intersect near the origin.\footnote{If $\beta>\gamma$, then $ c^{-1} b  \, x_j > x_i^{\beta / \gamma} $ defines a thick cusp. Otherwise, we have the thin cusp $c \, b^{-\gamma / \beta} x_i < x_j^{\gamma / \beta}$ .} Substituting independently $b x_j  = x_k^{\beta}$ and $c x_i = x_k^{\gamma}$ into 
$ c^{-1} b  \, x_j > x_i^{\beta / \gamma} $, we get $c x_i < x_k^{\gamma}$ and $ b x_j > x_k^{\beta}$. Hence, the simultaneous intersection of two thin cusps $V_{ji}(a,\alpha)$ and $V_{ik}(c,\gamma)$ with one thick cusp $V^\text{c}_{jk}(b,\beta)$
in the original space $\R^3_+$ is empty sufficiently close to the origin. See Figure~\ref{fig:n-pair-cusps-2}.

\begin{figure}[!htb]
\centering{}
\includegraphics[scale=0.65]{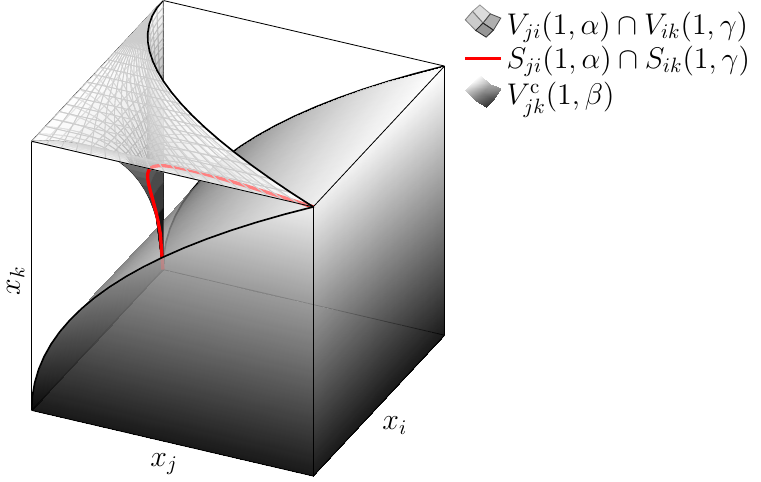}
\caption{In $\R^3$ the three cusps $V_{j i}(a,\alpha)$, $V^{\text{c}}_{jk}(b,\beta)$ and $V_{i k}(c,\gamma)$ do not intersect near the origin in $\R^3$ if $\beta<\alpha \gamma$ $(a=b=c=1)$.}
\label{fig:n-pair-cusps-2}
\end{figure}

\item[(ii)] $\beta > \alpha \gamma$.
In the $(x_i,x_j)$-subspace, the base step establishes that the two cusps $a x_j > x_i^{\alpha}$ and 
$ c^{-1} b  \, x_j < x_i^{\beta / \gamma} $ do not intersect near the origin. Substituting independently $b x_j  = x_k^{\beta}$ and $c x_i = x_k^{\gamma}$ into 
$ c^{-1} b  \, x_j < x_i^{\beta / \gamma} $, we get $c x_i > x_k^{\gamma}$ and $ b x_j < x_k^{\beta}$. Hence, the simultaneous intersection of two thick cusps $V^{\text{c}}_{ji}(a,\alpha)$ and $V^{\text{c}}_{ik}(c,\gamma)$ with one thin cusp $V_{jk}(b,\beta)$ 
in the original space $\R^3_+$ is empty sufficiently close to the origin. See Figure~\ref{fig:n-pair-cusps-3}.

\begin{figure}[!htb]
\centering{}
\includegraphics[scale=0.65]{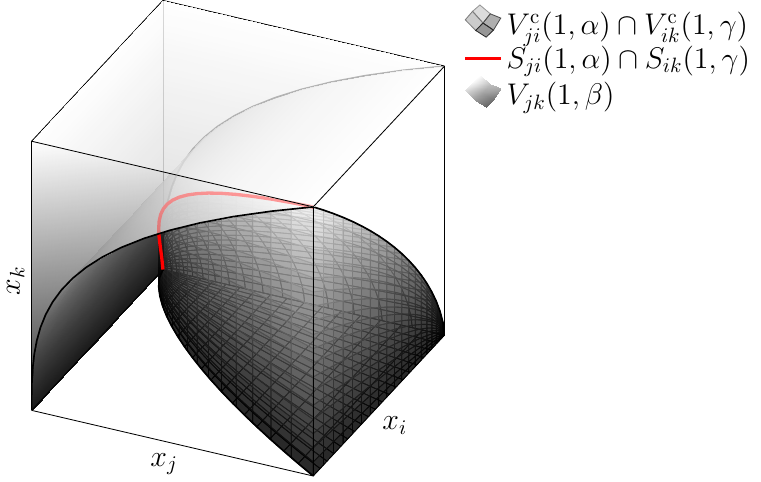}
\caption{In $\R^3$ the three cusps $V^{\text{c}}_{j i}(a,\alpha)$, $V_{jk}(b,\beta)$ and $V^{\text{c}}_{i k}(c,\gamma)$ do not intersect near the origin in $\R^3$ if $\beta>\alpha \gamma$ $(a=b=c=1)$.}
\label{fig:n-pair-cusps-3}
\end{figure}
\end{enumerate}

In higher dimensions, assume that the result holds for all $2 \leq n \leq m$, where $m\geq2$.
For $n=m+1$, we solve the system defined by $m+1$ arbitrary hypersurfaces by means of the elimination method. This allows us to restrict to a subspace of dimension lower than $m+1$ and apply the inductive hypothesis.
\end{proof}

Note that Proposition~\ref{lem:n-cusps} provides a bound on the number of intersecting pairs of cusps which is at least as good as Corollary~\ref{lem:cusps}, and strictly better if $n \geq 3$.
The hypotheses used in the two statements are different.

\subsection{Absence of infinite switching}\label{sec:no-infinite-switching}

For a quasi-simple heteroclinic cycle the expression for local and global maps was derived in~\cite{GdSC}. We give here a brief description in the context of quasi-simple heteroclinic networks. Let $\xi_j$ be a node in a quasi-simple heteroclinic network~\textbf{\textsf{N}}. 
The contracting eigenvalues\footnote{In what follows, unless stated otherwise, the eigenvalues are classified globally.} of $\textnormal{d}f(\xi_j)$ are denoted by $-c_{j,s}<0$, $s=1,\ldots,n_{c_j}$, the expanding ones by $e_{j,p}>0$, $p=1,\ldots,n_{e_j}$, and the transverse ones\footnote{Without loss of generality, all (globally) transverse eigenvalues are assumed to be negative. A necessary condition for the asymptotic stability of a heteroclinic network is that all transverse eigenvalues are negative. The existence of positive transverse eigenvalues provides a way for trajectories to escape a neighbourhood of the heteroclinic network, making switching unlikely for reasons other than the one we are interested in, namely, the real nature of the eigenvalues.} by $-t_{j,q} < 0$, $q=1,\ldots,n_{t_j}$. Of course, $n_{c_j} \geq 1$, $n_{e_j} \geq 1$ and $n_{t_j} \geq 0$. Accordingly, there are $n_{c_j}$~incoming connections to~$\xi_j$ and $n_{e_j}$~outgoing connections from~$\xi_j$. 

The definition of a quasi-simple heteroclinic network ensures that all cross sections to the flow along its heteroclinic connections have the same dimension, which in turn can be restricted to $\R^N$ with 
\begin{equation}
\label{eq:dim-cross}
N=n_{c_j}+n_{e_j}+n_{t_j}-1.
\end{equation}

Assume that $-c_{j,1}$ and $e_{j,1}$ are respectively the 
local-contracting and local-expanding eigenvalues 
at $\xi_j$ regarding the heteroclinic path $\left[ \xi_{i} \rightarrow \xi_j \rightarrow \xi_{k} \right]$. 
Let $v$, $w$, $\boldsymbol{z}$ be the local coordinates in the basis comprised of the associated local-contracting, local-expanding and local-transverse eigenvectors. In particular, $v$ and $w$ are 1-dimensional while $\boldsymbol{z}$ is 
$(N-1)$-dimensional. The vector $\boldsymbol{z}$ is divided into three subsets of coordinates as $\boldsymbol{z} \equiv (\boldsymbol{z}_c, \boldsymbol{z}_e, \boldsymbol{z}_t)$ 
representing respectively contracting, expanding and transverse directions. 

We use the coordinates $\left(w^{\text{in}},\boldsymbol{z}^{\text{in}}\right)$ in~$H_j^{\text{in},i}$ and $ \left(v^{\text{out}},\boldsymbol{z}^{\text{out}}\right)$ in~$H_j^{\text{out},k}$ so that
\[
\begin{aligned}
H_j^{\text{in},i} & \equiv \left\{\left(w^{\text{in}},\boldsymbol{z}^{\text{in}}\right) \in \R^N: \,\, 0 \leq \max_{l=1,\ldots,N-1} \{ w^{\text{in}}, z_l^{\text{in}} \} <1 \right\} \\
H_j^{\text{out},k} & \equiv \left\{\left(v^{\text{out}},\boldsymbol{z}^{\text{out}}\right) \in \R^N: \,\, 0 \leq \max_{l=1,\ldots,N-1} \{ v^{\text{out}}, z_l^{\text{out}} \} <1 \right\}.
\end{aligned}
\]
As usual, the local map 
$\phi_{ijk}:H_j^{\text{in},i} \rightarrow H_j^{\text{out},k}$ 
near $\xi_j$ is
calculated using the linearisation of the flow near $\xi_j$ and 
written down to leading order as
\[
\begin{aligned}
\left(v^{\text{out}},\boldsymbol{z}^{\text{out}}\right) & = \phi_{ijk}\left(w^{\text{in}},\boldsymbol{z}^{\text{in}}\right) \\
& = \left((w^{\text{in}})^{\frac{c_{j,1}}{e_{j,1}}}, \Big( z^{\text{in}}_{c,s} (w^{\text{in}})^{\frac{c_{j,s}}{e_{j,1}}}, z^{\text{in}}_{e,p} (w^{\text{in}})^{-\frac{e_{j,p}}{e_{j,1}}}, z^{\text{in}}_{t,q} (w^{\text{in}})^{\frac{t_{j,q}}{e_{j,1}}} \Big) \right) \\
\end{aligned}
\]
for $s=2,\ldots,n_{c_j}$, $p=2,\ldots,n_{e_j}$ and $q=1\ldots,n_{t_j}$.
More precisely, the domain of definition of $\phi_{ijk}$ is given by\footnote{The domain of $\phi_{ijk}$ is the set of $(w^{\text{in}},\boldsymbol{z}^{\text{in}}) \in H_j^{\text{in},i}$ for which $\phi_{ijk}(w^{\text{in}},\boldsymbol{z}^{\text{in}}) \in H_j^{\text{out},k}$. This is obtained by imposing that all the coordinates of $\phi_{ijk}$ are smaller than 1. Strictly speaking, there will be points in $H_i^{\text{in},j}$ that are excluded from the domain of $\phi_{ijk}$ (see, e.g., \cite{KS}). For simplicity, we ignore them because they do not enter into significant terms of shadow trajectories.} 
\[
C_{ijk} = \left\{ \left(w^{\text{in}},\boldsymbol{z}^{\text{in}}\right) \in H_j^{\text{in},i}: \,\,  z^{\text{in}}_{e,p} < (w^{\text{in}})^{\frac{e_{j,p}}{e_{j,1}}} \textnormal{ for all } p=2,\ldots,n_{e_j}\right\} 
 \]
and its image is given by\footnote{
The image of $\phi_{ijk}$ is the set of $(v^{\text{out}}, \boldsymbol{z}^{\text{out}}) \in H_j^{\text{out},k}$ whose pre-image under $\phi_{ijk}$ is in $C_{ijk}$.}
\[
\begin{aligned}[c] F_{ijk} =\bigg\{ \left(v^{\text{out}},\boldsymbol{z}^{\text{out}}\right) \in H_j^{\text{out},k}: & \,\,  z^{\text{out}}_{c,s} < (v^{\text{out}})^{\frac{c_{j,s}}{c_{j,1}}} \textnormal{ for all } s=2,\ldots,n_{c_j} \\
\text{ and } & \,\,  z^{\text{out}}_{t,q} < (v^{\text{out}})^{\frac{t_{j,q}}{c_{j,1}}} \textnormal{ for all } q=1,\ldots,n_{t_j} 
\bigg\}.
\end{aligned}
\]
If $n_{e_j} =1$, then $C_{ijk} = H_j^{\text{in},i}$. If $n_{c_j} = 1$ and $n_{t_j}=0$, then $F_{ijk}=H_j^{\text{out},k}$.

We see that the sets $C_{ijk}$ and $F_{ijk}$ are, respectively, the intersection of $n_{e_j}-1$ and $n_{c_j} + n_{t_j} -1$ cusps, 
thick or thin depending on the ratio of 
eigenvalues.  The symmetry axis in $C_{ijk}$ is $w^{\text{in}}$ and that in $F_{ijk}$ is $v^{\text{out}}$.
The cross sections $H_j^{\text{in},i}$ and $H_j^{\text{out},k}$ 
contain
pairs of cusps. Intersections of some of these cusps form $C_{ijk}$ and $F_{ijk}$, respectively. 
They delineate the heteroclinic paths that go through the node $\xi_j$ in the heteroclinic network.


\begin{figure}[!htb]
\centering
\includegraphics{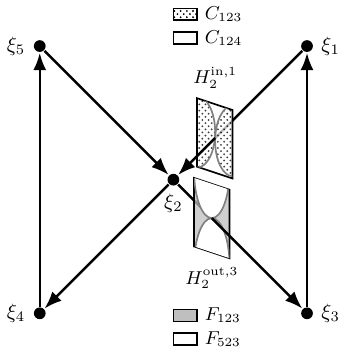}
\caption{The bowtie network from \cite{CastroLohse2016}. The sets $C_{i2k}$ are in $H_2^{\text{in},1}$ and $F_{i23}$ are in $H_2^{\text{out},3}$.}
\label{fig:bowtie}
\end{figure}

As an illustration, consider the bowtie network in \cite{CastroLohse2016} depicted in Figure~\ref{fig:bowtie}: near $\xi_2$ the set $C_{123}$ 
describes the set of points following the heteroclinic path $\left[ \xi_{1} \rightarrow \xi_2 \rightarrow \xi_{3} \right]$, whereas its complement in $H_2^{\text{in},1}$, $C_{124}$, 
describes those points that follow $\left[ \xi_{1} \rightarrow \xi_2 \rightarrow \xi_{4} \right]$. Analogously, 
$F_{123}$ 
describes the set of points that follow the heteroclinic path $\left[ \xi_{1} \rightarrow \xi_2 \rightarrow \xi_{3} \right]$,  
whereas its complement in~$H_2^{\text{out},3}$,~$F_{523}$, 
describes those points that follow 
$\left[ \xi_{5} \rightarrow \xi_2 \rightarrow \xi_{3} \right]$.

While the cross sections $H_j^{\text{in},i}$ are always the union of pairs of sets $C_{ijk}$ (except for very thin cusps containing points that are removed from near the heteroclinic network), this holds for $H_j^{\text{out},k}$ and the sets $F_{ijk}$ only if $n_{t_j}=0$. Consider, for example, the house network from \cite{CastroLohse2016}, studied in \cite{tese} from which we reproduce Figures~4.12 
and~4.13~(b) in Figure~\ref{fig:House}.
\footnote{
Comparing with Figure~\ref{fig:House}~(b), the sets $F_{i12}$ in Figure~6 in~\cite{CastroLohse2016} are not well-defined. 
The results in \cite{tese} correct Lemma 3.7 in \cite{CastroLohse2016}, showing that there is no switching near the house network. This correction also follows from Proposition~\ref{prop:swit-con} below.}
Here we have $n_{t_1}=1>0$ at the node $\xi_1$ so that the outgoing cross section $H_1^{\text{out},2}$ is more than the union of the sets $F_{i12}$. The points in $H_{1}^{\text{out},2}\setminus F_{312} \cup F_{512}$ do not remain close to the network.

\begin{figure}[!htb]
\centering
\subfloat[]{\includegraphics{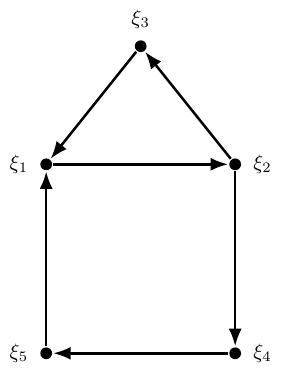}}
\hfill
\subfloat[]{\includegraphics[scale=0.65]{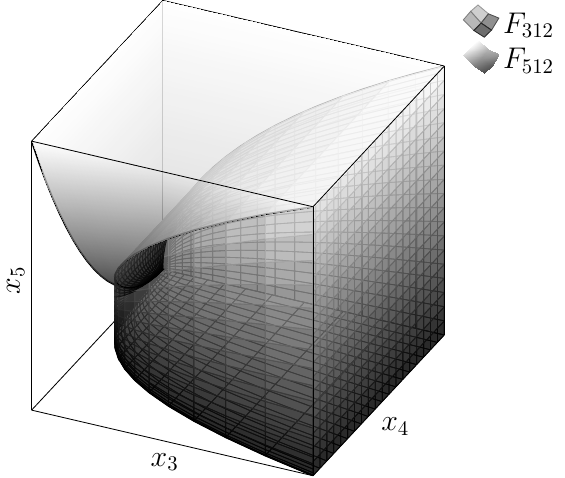}}
\caption{(a) The house network from \cite{CastroLohse2016}. (b) The cross section $H_1^{\text{out},2}$ is not the union of the sets $F_{312}$ and $F_{512}$.}
\label{fig:House}
\end{figure}

As in \cite{GdSC} we assume that the quasi-simple heteroclinic networks are such that the global map $\psi_{jk}:H_j^{\text{out},k} \rightarrow H_k^{\text{in},j}$ along each heteroclinic connection $\left[ \xi_j \rightarrow \xi_{k} \right]$ is a rescaled permutation. 
Accordingly, every cusp is mapped by $\psi_{jk}$ into a new cusp.

The nature of the sets $C_{ijk}$, as well as that of their images, determines the absence of infinite switching as we prove next.

\begin{theorem}\label{th:no-switching}
In quasi-simple heteroclinic networks such that at all nodes the Jacobian matrix has only real eigenvalues there is no infinite switching.
\end{theorem}

\begin{proof}
Let $\xi_j$ be a distribution node. Then, $H_j^{\text{in},i}$ is made of at least one pair of cusps with symmetry axis $w^{\text{in}}$. Recall that the sets $C_{ijk}$ covering $H_j^{\text{in},i}$ are the intersection of a number of cusps which depends on the number of outgoing connections from the distribution node. Since all eigenvalues are real, the image of the sets $C_{ijk}$ in $H_j^{\text{in},i}$ under the full return map to $H_j^{\text{in},i}$ is again the intersection of cusps.

Using Proposition~\ref{lem:n-cusps}, we know that only a finite number of cusps, limited by the dimension of $H_j^{\text{in},i}$, can have all-to-all intersection. 
The fact that infinite switching requires an all-to-all intersection of an infinite number of cusps finishes the proof. 
\end{proof}

An immediate consequence is that, because no infinite switching occurs, chaotic behaviour cannot be observed near a quasi-simple heteroclinic network. If the heteroclinic network is asymptotically stable, then no chaos exists for the dynamics since all trajectories are attracted to the network. On the other hand, when the heteroclinic network is not asymptotically stable, chaos may appear only at a distance from the network.

Note that the assumption about the sign of transverse eigenvalues only affects the definitions of the sets $C_{ijk}$ and $F_{ijk}$.
Negative transverse eigenvalues are necessary if we want to admit asymptotic stability of the heteroclinic network. If there exists one positive transverse eigenvalue then the unstable manifold of the corresponding equilibrium is not contained in the heteroclinic network. In this case, by Theorem 2.8 in \cite{PCL2020}, the heteroclinic network is not asymptotically stable, making infinite switching less likely (if its existence were possible).  

\subsection{Constraints on switching}\label{sec:switching}

Having proved the absence of infinite switching, we now 
specify in more detail ways by which infinite switching does not occur. In particular, a closer look at the geometry of local and global maps near a sequence of heteroclinic connections provides a description of how rich the switching dynamics can be. Namely, it provides an upper bound for the number of 
distribution nodes that exist in a heteroclinic path
that may be followed from a given node, as a function of the dimension of the state space and the number of transverse eigenvalues.

The next result shows that when the number of either incoming or outgoing connections at a node $\xi_j$ equals the dimension of the cross sections, it is not possible to return to near $\xi_j$ and leave a neighbourhood of $\xi_j$ in all possible combinations of incoming and outgoing paths. 
This yields an alternative proof of the absence of switching in the Kirk and Silber network
to~\cite{CastroLohse2016}: 
the cross sections are 2-dimensional, at the node $\xi_1$ it is $n_{c_1}=2$ and at $\xi_2$ we have $n_{e_2}=2$.
An example in $\R^5$ is the ac-heteroclinic network from~\cite{PCL2020},
where at $\xi_3$ it is $n_{c_3}=3$.
An example in $\R^3$ is given in~\cite{PeiRod2022},
where 2-dimensional cross sections are introduced at every node such that either $n_{c_j}=2$ or $n_{e_j}=2$. See Figure~\ref{fig:KS_ac}.

\begin{figure}[!htb]
\centering
\subfloat[]{\includegraphics{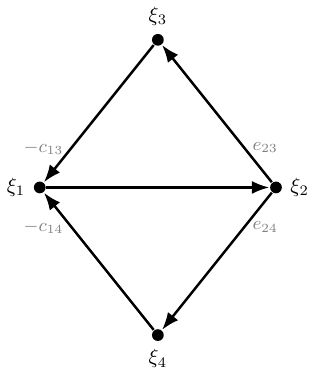}} 
\hfill
\subfloat[]{\includegraphics{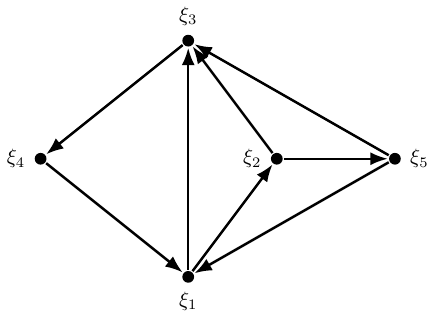}} 

\subfloat[]{\includegraphics{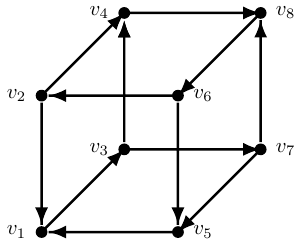}} 
\caption{The heteroclinic networks (a) from~\cite{KS}, (b) from~\cite{PCL2020} and (c)~from~\cite{PeiRod2022}.}
\label{fig:KS_ac}
\end{figure}

Recall that $N$ represents the dimension of the cross sections to the flow. 

\begin{proposition}\label{prop:swit-node}
For a quasi-simple heteroclinic network \textbf{\textsf{N}}  in $\R^n$, if there exists some node~$\xi_j$ such that either $n_{c_j}=N$ or $n_{e_j}=N$, then there is no (finite or infinite) switching near \textbf{\textsf{N}}.
\end{proposition}

\begin{proof}
Assume that $n_{e_j}=N$. The proof for $n_{c_j}=N$ is analogous by looking at $H_j^{\text{out},k}$.

By virtue of~\eqref{eq:dim-cross} we have $ n_{c_j} + n_{t_j} = 1$. A quasi-simple heteroclinic network admits $n_{c_j} \geq 1$, which implies $n_{t_j}=0$.
The cross section $H_j^{\text{in},i}$ consists in 
$n_{e_j}=N$ pairs of cusps. We look at points in $H_j^{\text{in},i}$ and their images under a full return map, along any heteroclinic cycle, to this cross section.
The existence of infinite switching demands that the 
$n_{e_j}=N$ pairs of cusps in $H_j^{\text{in},i}$ intersect with their images, that is, there is infinite switching only if 
$2n_{e_j}=2N \ (>N)$ pairs of cusps intersect in $H_j^{\text{in},i}$. Given that the dimension of $H_j^{\text{in},i}$ is equal to $N$, it follows from Proposition~\ref{lem:n-cusps} that this is impossible, since for a quasi-simple heteroclinic network the dimension of a cross section is always greater than~1.
\end{proof}

From 
the previous proof, we obtain a 
thorough description of the possible choices of paths around a heteroclinic cycle in the heteroclinic network:
if a distribution node satisfies one of the hypotheses of Proposition~\ref{prop:swit-node}, then after 
at most one return to this node it is no longer possible to proceed along all available outgoing paths. 

The next result shows that, when two heteroclinic cycles in the heteroclinic network 
have a common connection $\left[ \xi_1 \rightarrow \xi_2 \right]$ and the number of incoming and outgoing connections from $\left[ \xi_1 \rightarrow \xi_2 \right]$ is at least the dimension of the cross sections, there is no switching along the common connection. Its corollary extends the absence of switching to a common sequence of connections. This provides yet another proof of the absence of switching in the Kirk and Silber network of \cite{KS}, this time along 
the 
connection shared by
the two heteroclinic cycles.
Other examples may be found in the heteroclinic network of the 2-person Rock-Scissors-Paper (RSP) game studied
in \cite{GSC_RSP} and in the heteroclinic network of the Rock-Scissors-Paper-Lizard-Spock (RSPLS) game
of \cite{CFGdSL}, 
see Figure~\ref{fig:RSP}.

\begin{proposition}\label{prop:swit-con}
Suppose that, in a quasi-simple heteroclinic network \textbf{\textsf{N}}  in $\R^n$, two heteroclinic cycles share a heteroclinic connection $\left[ \xi_1 \rightarrow \xi_2 \right]$.
If $n_{c_1} + n_{e_2} \geq N$, then there is no (finite or infinite) switching near \textbf{\textsf{N}}. 
\end{proposition}

\begin{proof}
We show that there  is no switching near \textbf{\textsf{N}} by proving that there is no switching along the connection $\left[ \xi_1 \rightarrow \xi_2 \right]$.

Assume for now that $n_{t_1}=0$. Then the
cross section $H_1^{\text{out},2}$ is the union of $n_{c_1}$ sets $F_{i12}$, one for each incoming connection at $\xi_1$. Analogously, the 
cross section $H_2^{\text{in},1}$ is the union of $n_{e_2}$ sets $C_{12k}$, one for each outgoing connection from $\xi_2$. Note that when $n_{c_1}=1$ there is only one set $F_{i12}$ which is the whole of $H_1^{\text{out},2}$. If $n_{c_1}=2$ there are two sets $F_{i12}$ and they are a pair of cusps. When $n_{c_1}>2$ the $n_{c_1}$ sets $F_{i12}$ are contained in nested cusps.

According to Definition~\ref{def:swit-con}, there is switching along a connection if the images of the $n_{c_1}$ sets $F_{i12}$ under the global map along the connection intersect all-to-all the $n_{e_2}$ sets $C_{12k}$.

For quasi-simple cycles the global map is a rescaled permutation of the local coordinate axes, ensuring that the intersection of the images of the $n_{c_1}$ sets $F_{i12}$ with the $n_{e_2}$ sets $C_{12k}$ is contained in the intersection of $n_{c_1}+n_{e_2}$ pairs of cusps. By hypothesis $n_{c_1}+n_{e_2}\geq N$ and, from Proposition~\ref{lem:n-cusps}, at least two of the cusps do not intersect. 
Hence, there is no switching along the connection.

If $n_{t_1}>0$ then the sets $F_{i12}$ are smaller than if $n_{t_1}=0$ and the result is proved.
\end{proof}

Since $n_{c_1}\geq 1$ and $n_{e_2}\geq 1$, it follows trivially that if the dimension of the cross sections is $N=2$, then there is no switching along any heteroclinic connection.


\begin{corollary}\label{cor:sequence}
If, in Proposition~\ref{prop:swit-con}, the heteroclinic connection $\left[\xi_1\rightarrow \xi_2\right]$ is replaced by $\left[\xi_1 \rightarrow \xi_{j_1} \rightarrow \xi_{j_2} \cdots \rightarrow \xi_{j_k} \rightarrow \xi_2 \right]$, with $k \in \N$ non-distribution nodes $\xi_{j_i}$, then there is no (finite or infinite) switching near \textbf{\textsf{N}}.
\end{corollary}

Notice that the information provided by Propositions~\ref{prop:swit-node} and \ref{prop:swit-con} when applied to the Kirk and Silber network in \cite{KS} is not equivalent. Proposition~\ref{prop:swit-node} informs that after taking a turn around any of the heteroclinic cycles, the outgoing available options are no longer all feasible, whereas Proposition~\ref{prop:swit-con} shows that the limitation in choice occurs along the heteroclinic connection $\left[\xi_1\rightarrow \xi_2\right]$: it is not possible to arrive at a neighbourhood of $\xi_1$ and leave a neighbourhood of $\xi_2$ in all available combinations.


The next result provides 
an upper bound for the number of distribution nodes that can exist in a sequence of heteroclinic connections for which there is switching. It applies to sequences of heteroclinic connections involving multiple distribution nodes for which Corollary~\ref{cor:sequence} is inconclusive.
Let there be a distribution node, without loss of generality, $\xi_{j_0}$, in the heteroclinic network. A sequence of connections from $\xi_{j_0}$ with $m$ distribution nodes is $\left[\xi_{j_0} \rightarrow \cdots \rightarrow \xi_{j_1} \rightarrow \cdots \rightarrow \xi_{j_2} \rightarrow \cdots \rightarrow \xi_{j_m} \right]$. 
Note that there may be other (non-distribution) nodes in between the distribution nodes $\xi_{j_i}$ with $i=1, \ldots, m$.

\begin{theorem}\label{th:finite-switching}
In $\R^n$ let $\xi_{j_0}$ be a distribution node in a quasi-simple heteroclinic network.
If there is switching along a heteroclinic path from $\xi_{j_0}$ with $m\geq 0$ additional distribution nodes
then $\sum_{i=1}^m (n_{c_{j_{i-1}}}+n_{e_{j_i}}) < N$.
\end{theorem}

\begin{proof}
The inequality is obtained by proceeding as in the proof of Proposition~\ref{prop:swit-con} in a contrapositive way. The upper bound of the sum follows from Proposition~\ref{prop:swit-node}.
\end{proof}

%

We end the section with some illustrative examples. 

\begin{example}[The house network in \cite{CastroLohse2016}]
The common connection $[\xi_1 \rightarrow \xi_2]$ of the house network in Figure~\ref{fig:House} satisfies $n_{c_1}=2$ and $n_{e_2} =2$. The dimension of the cross sections is $N=3$. Since $n_{c_1}+ n_{e_2}=2+2>3=N$, Proposition~\ref{prop:swit-con} guarantees no switching near this heteroclinic network. 

Such a result rectifies Lemma 3.7 in \cite{CastroLohse2016} and is consistent with the correction in~\cite{tese} that was first made to the intersection of the sets $F_{i12}$ with $C_{12k}$ in~\cite{CastroLohse2016}. 
\end{example} 

\begin{example}[The 2-person RSP game of \cite{GSC_RSP}]
The quotient RSP network 
reproduced in 
Figure~\ref{fig:RSP}~(a)
exists in $\R^4$. 
Cross sections have dimension
$N=3$. 
At any node $\xi_i$ we see that $n_{c_{i-1}}=2$ and $n_{e_i}=2$. 
For $k=1$ in Theorem~\ref{th:finite-switching}, we have $n_{c_{i-1}}+n_{e_i}>N$ and thus one is the maximum number of distribution nodes that is reached from any starting node. Hence, there is no switching along any of the heteroclinic connections in this network. 

Note that we could also have applied Proposition~\ref{prop:swit-con}.
A proof of 
the absence of switching by direct (and long) computation can be found in~\cite{Cesary}.

We include this example to correct a claim to infinite switching made in \cite{AC} due to the use of inappropriate global maps (see \cite{GSC_RSP} for a correct version of the global maps).

The same conclusion can be drawn for the extension of the RSP by introducing two more possibilities, Lizard and Spock, where the RSPLS network in Figure~\ref{fig:RSP}~(b) arises in $\R^5$. Even though \cite{PosRuc2022} report on various finite heteroclinic paths (root sequences) that can be observed near the RSPLS network, just not for the same parameter values, our results state that not all possible finite heteroclinic paths are followed regardless parameter values.

\begin{figure}[!htb]
\centering
\subfloat[]{\includegraphics{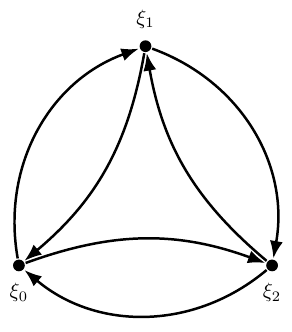}}
\hfill
\subfloat[]{\includegraphics{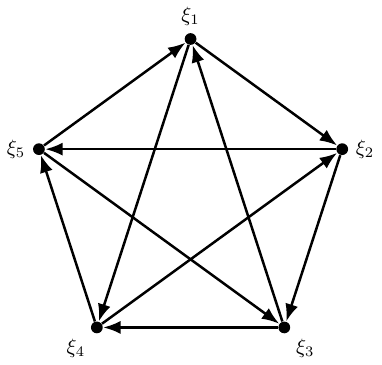}}
\caption{The heteroclinic networks (a) from~\cite{GSC_RSP} and (b) from~\cite{CFGdSL}. All nodes are distribution nodes.}
\label{fig:RSP}
\end{figure}

\end{example}

\begin{example}[The bowtie network in \cite{CastroLohse2016}]
The lowest dimension for the construction of the bowtie network in Figure~\ref{fig:bowtie} is $\R^5$. Cross sections are 3-dimensional 
and only $\xi_2$ is a distribution node. 
From Theorem~\ref{th:finite-switching} we obtain $n_{c_{j_{i-1}}}+n_{e_{j_{i}}}=n_{c_2}+n_{e_2}=2+2>N=3$. 
Therefore, trajectories starting at one of the incoming cross sections of $\xi_2$ will have a predetermined outgoing cross section after one return.
\end{example}

\begin{example}\label{ex:triangle}
The heteroclinic network in Figure~\ref{fig:doubleRSP} can be generated in~$\R^6$ with the simplex method from~\cite{AP}. Cross sections transverse to the flow are reduced to dimension $N=4$.
The distribution nodes are~$\xi_1$,~$\xi_2$ and~$\xi_6$.

At $\xi_1$ we have $n_{c_2}+n_{e_1}=1+2 < N$, and at $\xi_2$ we have $n_{c_2}+n_{e_6}=1+2 < N$.  
But, in Theorem~\ref{th:finite-switching}, $\sum_{i=1}^2 (n_{c_{j_{i-1}}}+n_{e_{j_{i}}}) = 6 \geq N$.
Accordingly, trajectories starting at the incoming cross section of $\xi_1$ can follow any outgoing cross section of $\xi_2$ but will have a predetermined outgoing cross section at the next distribution node (this might be $\xi_6$ or again $\xi_1$).
The sequences with two distribution nodes from $\xi_1$ are
$\left[\xi_1 \rightarrow \xi_{2} \rightarrow \xi_5 \rightarrow \xi_{6} \right]$,
$\left[\xi_1 \rightarrow \xi_{2} \rightarrow \xi_3 \rightarrow \xi_{1} \right]$
and
$\left[\xi_1 \rightarrow \xi_{4} \rightarrow \xi_5 \rightarrow \xi_{6} \right]$.

Analogous behaviour occurs at the other distribution nodes.

\begin{figure}[!htb]
\centering
\includegraphics{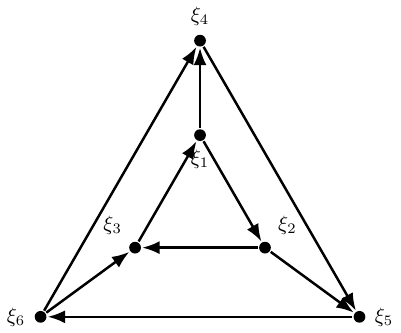}
\caption{A heteroclinic network in $\R^6$. The nodes $\xi_1$, $\xi_2$ and $\xi_6$ are distribution nodes.}
\label{fig:doubleRSP}
\end{figure}
\end{example}


\section{Final remarks}\label{sec:final}

The complexity of the behaviour in a neighbourhood of a heteroclinic network depends on the type of eigenvalue, real or complex, of the Jacobian matrix at each node. It is known that complex eigenvalues induce infinite switching and therefore chaos, but the novelty is that real eigenvalues can lead to rich, but not chaotic, dynamics near the heteroclinic network.

While we have established the absence of infinite switching near a generic class of heteroclinic networks, we hope to have also encouraged the scientific community to look for different types of finite switching. The examples at the end of the previous section show that there are several ways in which infinite switching fails to exist. 
Several forms of finite switching exist, that is, a wide variety of finite heteroclinic paths can be followed by trajectories near the RSPLS network for a specific set of parameter values such that the heteroclinic cycles in the network are unstable. See \cite{CFGdSL} and \cite{PosRuc2022}.

Furthermore, chaotic behaviour might be observed at a distance from the heteroclinic network. This may well be attracting for trajectories starting 
close the heteroclinic network provided 
the latter is not asymptotically stable. One such 
case is 
found by Sato {\em et al.}~\cite{SAC} and depicted in their Figure~13. The parameter values for this figure belong to a region where none of the heteroclinic cycles possesses any stability as determined in~\cite{GSC_RSP}. 

\paragraph{Acknowledgements:} The authors thank A.\ Lohse for his comments on a previous version.

Both authors were partially supported by CMUP, member of LASI, which is financed by national funds through FCT -- Funda\c{c}\~ao para a Ci\^encia e a Tecnologia, I.P., under the project reference UID/MAT/00144/2020. Some of the work was done while L.\ Garrido-da-Silva was the recipient of the doctoral grant PD/BD/105731/2014 from FCT (Portugal).


\begin{thebibliography}{99}

\bibitem{AfrMosYou}
Afraimovich V S, Moses G and Young T (2016)
Two-dimensional heteroclinic attractor in the generalized Lotka-Volterra system. 
\emph{Nonlinearity}~{\bf 29}, 1645--1667.

\bibitem{Aguiar2011}
Aguiar M A D (2011)
Is there switching in bimatrix games?
\emph{Physica D}~{\bf 240}~(18), 1475--1488.

\bibitem{AC}
Aguiar M A D and Castro S B S D (2010)
Chaotic switching in a two-person game. 
\emph{Physica D}~{\bf 239}, 1598--1609.

\bibitem{ACL2005}
Aguiar M A D, Castro S B S D and Labouriau I S (2005)
Dynamics near a heteroclinic network.
\emph{Nonlinearity}~{\bf 18}~(1), 391--414.

\bibitem{ACL2006}
Aguiar M A D, Castro S B S D and Labouriau I S (2006)
Simple vector fields with complex behaviour.
\emph{Int. J. Bifurc. Chaos}~{\bf 16}~2, 369--381.
 
\bibitem{AP}
Ashwin P and Postlethwaite C (2013)
On designing heteroclinic networks from graphs. 
\emph{Physica D}~{\bf 265}~(1), 26--39.

\bibitem{Bic2018}
Bick C (2018)
Heteroclinic switching between chimeras.
\emph{Phys. Rev. E}~{\bf 97}~(5), 050201(R).

\bibitem{BicLoh2019}
Bick C and Lohse A (2019)
Heteroclinic Dynamics of Localized Frequency Synchrony: Stability of Heteroclinic Cycles and Networks.
\emph{J. Nonlinear Sci.}~{\bf 29}, 2571--2600.

\bibitem{CFGdSL}
Castro S B S D, Ferreira A, Garrido-da-Silva L, Labouriau I S. 
Stability of cycles in a game of Rock-Scissors-Paper-Lizard-Spock.
\emph{SIAM J. Appl. Dyn. Syst.}, to appear.

\bibitem{CastroLohse2016}
Castro S B S D, Lohse A (2016)
Switching in heteroclinic networks. 
\emph{SIAM J. Appl. Dyn. Syst.}~{\bf 15}~(2), 1085--1103.

\bibitem{tese}
Garrido-da-Silva L (2018)
Heteroclinic Dynamics in Game Theory. 
PhD thesis. University of Porto.

\bibitem{GdSC}
Garrido-da-Silva L and Castro S B S D (2019)
Stability of quasi-simple heteroclinic cycles.
\emph{Dyn. Syst.}~{\bf 34}~(1), 14--39.

\bibitem{GSC_RSP}
Garrido-da-Silva L and Castro S B S D (2020)
Cyclic dominance in a two-person rock-scissors-paper game. 
\emph{Int. J. Game Theory}~{\bf 49}, 885--912.

\bibitem{KS}
Kirk V, Silber M (1994)
A competition between heteroclinic cycles.
\emph{Nonlinearity}~{\bf 7}, 1605--1621.

\bibitem{KrupaMelbourne2004}
Krupa M and Melbourne I (2004)
Asymptotic stability of heteroclinic cycles in systems with symmetry II.
\emph{Proc Royal Soc. Edin.}~{\bf 134}, 1177--1197.

\bibitem{LabRod2017}
Labouriau I S and Rodrigues A A P (2017)
On Takens' Last Problem: tangencies and time averages near heteroclinic networks.
\emph{Nonlinearity}~{\bf 30}, 1876--1910

\bibitem{Cesary}
Olszowiec C (2016)
Complex behaviour in cyclic competition bimatrix games.
\emph{arXiv:1605.00431v4}.

\bibitem{PeiRod2022}
Peixe T and Rodrigues A A P (2022)
Stability of heteroclinic cycles: a new approach.
\emph{arXiv:2204.00848}.

\bibitem{Podvigina2012}
Podvigina O (2012)
Stability and bifurcations of heteroclinic cycles of type $Z$.
\emph{Nonlinearity}~{\bf 25}, 1887--1917.

\bibitem{PCL2020}
Podvigina O, Castro S B S D and Labouriau I S (2020)
Asymptotic stability of robust heteroclinic networks.
\emph{Nonlinearity}~{\bf 33}, 1757--1788.

\bibitem{PodviginaLohse2019}
Podvigina O and Lohse A (2019)
Simple heteroclinic networks in $\R^4$.
\emph{Nonlinearity}~{\bf 32}, 3269--3293.

\bibitem{PosDaw2010}
Postlethwaite C M and Dawes J H P (2010)
Regular and irregular cycling near a heteroclinic network.
\emph{Nonlinearity}~18, 1477--1509.

\bibitem{PosRuc2022}
Postlethwaite C M and Rucklidge A M (2022)
Stability of cycling behaviour near a heteroclinic network model of Rock-Paper-Scissors-Lizard-Spock.
\emph{Nonlinearity}~35, 1702--1733.

\bibitem{Rab_etal2010}
Rabinovich M I, Afraimovich V S and Varona P (2010)
Heteroclinic binding.
\emph{Dyn. Syst.}-{\bf 25}~(3), 433--442.

\bibitem{RodLab2014}
Rodrigues A A P and Labouriau I S (2014)
Spiralling dynamics near heteroclinic networks.
\emph{Physica D}~{\bf 268}, 34--39.

\bibitem{SAC}
Sato Y, Akiyama E and Crutchfield J P (2005)
Stability and Diversity in Collective Adaptation.
\emph{Physica D}~{\bf  210}, 21--57.

\end{thebibliography}
\end{document}